\documentclass[12pt]{amsart}
\usepackage[english]{babel}
\usepackage{graphicx}
\usepackage{amsmath,amssymb,hyperref,srcltx}
\newcommand{\diff}[2]{\mbox{{\rm Diff}{${\,}_{#1}({\mathbb C}^{#2},0)$}}}
\newcommand{\diffh}[2]{\mbox{$\widehat{\rm Diff}{{\,}_{#1}({\mathbb C}^{#2},0)}$}}

\newcommand{\cn}[1]{\mbox{\rm{(}${\mathbb C}^{#1},0$\rm{)}}}

\newtheorem{pro}{Proposition}[section]
\newtheorem{teo}{Theorem}[section]
\newtheorem*{main}{Main Theorem}
\newtheorem{cor}{Corollary}[section]

\newtheorem{lem}{Lemma}[section]

\theoremstyle{remark}
\newtheorem{rem}{Remark}[section]

\theoremstyle{definition}
\newtheorem{defi}{Definition}[section]

\begin{document}

\title[Local intersection dynamics and multiplicities]
{Boundness of intersection numbers for actions by two-dimensional biholomorphisms}




\author{Javier Rib\'{o}n}
\address{Instituto de Matem\'{a}tica e Estat\'\i stica \\
Universidade Federal Fluminense\\
Campus do Gragoat\'a\\
Rua Marcos Valdemar de Freitas Reis s/n, 24210\,-\,201 Niter\'{o}i, Rio de Janeiro - Brasil }
\thanks{e-mail address: jribon@id.uff.br}
\thanks{MSC class. Primary: 32H50, 37C85; Secondary: 37F75,  20G20, 20F16}
\date{\today}
\maketitle

\bibliographystyle{plain}
\section*{Abstract}
We say that a group $G$ of local (maybe formal) biholomorphisms satisfies the uniform intersection
property if the intersection multiplicity $(\phi (V), W)$ takes only finitely many values as a function of
$G$ for any choice of analytic sets $V$ and $W$. In dimension $2$ we show that $G$ satisfies the
uniform intersection property if and only if it is finitely determined, i.e. there exists a natural number
$k$ such that different elements of $G$ have different Taylor expansions of degree $k$ at the origin.
We also prove that $G$ is finitely determined if and only if the action of $G$ on the space of germs of
analytic curves have discrete orbits.
\section{Introduction}
Consider a subgroup $G$ of the group $\diff{}{n}$ of
local biholomorphisms  defined in a neighborhood of the origin of
${\mathbb C}^{n}$. We are interested in understanding the properties of the action of $G$ on the
intersection multiplicities $(V,W)$ of analytic sets $V$ and $W$.
Given local complex analytic sets defined in a neighborhood of $0$ in ${\mathbb C}^{n}$
we study the behavior of the function $\phi \mapsto (\phi (V) ,W)$ defined in $G$.
This can be considered as a proxy for the study of the dynamics of the group and its mixing properties.
More precisely, we want to identify the groups such that the previous function only takes finitely many
values for any choice of $V$ and $W$. We say that such groups satisfy the Uniform Intersection (UI)
property.

A motivation to study the (UI) property is provided by the following result by Shub and Sullivan:
\begin{teo}[\cite{shub-sullivan}]
Let $f: U \to {\mathbb R}^{m}$ be a $C^{1}$ map where $U$ is an open subset of ${\mathbb R}^{m}$.
Suppose that $0$ is an isolated fixed point of $f^n$ for every $n \in {\mathbb N}$.
Then the sequence $(I(f^n,0))_{n \geq 1}$ of fixed point indexes is bounded by above by a constant that does not
depend on $n$.
\end{teo}
 Consider the map $F : U \times {\mathbb R}^{m} \to {\mathbb R}^{m} \times {\mathbb R}^{m}$ defined by the formula
$F(x,y)=(f(x),y)$. The fixed point index of $f^{n}$ at the origin can be interpreted as the topological
intersection multiplicity $(\Delta, F^{n} (\Delta))$ at $0 \in {\mathbb R}^{m} \times {\mathbb R}^{m}$
where $\Delta = \{ (x,x) : x \in {\mathbb R}^{m} \}$ is the diagonal.
Shub and Sullivan's theorem can be used to show that a a $C^{1}$ automorphism $f: M \to M$ of a compact manifold $M$,
whose sequence of Lefschetz numbers $(L(f^n))_{n \geq 1}$ is unbounded, has infinitely many periodic points.

\strut

The previous theorem inspired Arnold \cite{Arnold:intersection} to study the uniform intersection
property for $C^{\infty}$ and holomorphic maps and specially for diffeomorphisms.
Arnold's results were generalized by Seigal-Yakovenko \cite{Seigal-Yakovenko:ldi} and Binyamini \cite{Binyamini:finite}.
The degree of regularity is much higher than $C^{1}$ of course, but on the other hand the results can be applied to more general classes of maps and groups.

\strut

This paper completes in dimension $2$ the program of characterizing the groups that satisfy the uniform intersection
property, that was partially carried out in
\cite{Arnold:intersection, Seigal-Yakovenko:ldi, Binyamini:finite, JR:finite}
(see Main Theorem).

\strut

The paper is  inserted in a classical subject, namely
the local study of asymptotics of topological complexity of
intersections.
The global case has been treated extensively in the literature
\cite{Artin-Mazur} \cite{Arnold:complexity} \cite{Arnold:dynamics} \cite{Arnold:problems}
and references therein. Moreover, there are more than 15
problems in the book of Arnold's problems devoted to asymptotics of topological complexity
of intersections in both the local and global settings \cite{Arnold:problems}.

\strut

Let us recap some results   in the literature.
In the context of subgroups of $\diff{}{n}$ (or its formal completion $\diffh{}{n}$), with $n \geq 2$,
we have the following inclusions:

\begin{equation}
\label{equ:inc}
 \mathrm{cyclic} \subset \mathrm{finitely \ generated \ abelian} \subset \mathrm{Lie} =
\mathrm{F_{\dim}} \subset \mathrm{UI} \subset \mathrm{FD} .
\end{equation}
Cyclic groups satisfy (UI) by a Theorem of Arnold \cite{Arnold:intersection}.
The result was generalized to finitely generated abelian groups by Seigal and Yakovenko
\cite{Seigal-Yakovenko:ldi}. Binyamini showed (UI) for subgroups of $\diffh{}{n}$ that can be embedded
in a subgroup of $\diffh{}{n}$ with a natural structure of finitely dimensional Lie group with finitely
many connected components \cite{Binyamini:finite}. Moreover he showed that every finitely generated
abelian group admits such an embedding. The term ${\mathrm F}_{\dim}$ stands for the set of subgroups of
$\diffh{}{n}$ whose Zariski closure has finite dimension;
they satisfy (UI) by a theorem of the author \cite{JR:finite}.
These groups are exactly the groups
that can be embedded in Lie groups (with finitely many connected components)
but the use of intrinsic techniques allowed us to show for instance that
virtually polycyclic subgroups of $\diffh{}{n}$ (and in particular finitely generated nilpotent subgroups) hold
property (UI). Very roughly speaking, in all previous cases the problem is transferred to a finite dimensional
space where some finiteness result (the Skolem theorem on zeroes of quasipolynomials \cite{skolem:seq},
noetherianity arguments) is used.

The term (FD) in sequence (\ref{equ:inc}) stands for finite determination; we say that $G$ is finitely
determined if there exists $k \in {\mathbb N}$ such that every element of $G$ is determined by
the $k$ jet of its Taylor power series expansion at the origin.
It is easy to verify that (UI) groups are always finitely determined for $n \geq 2$
(Lemma \ref{lem:easy}).

Let us focus on the subsequence $\mathrm{F_{\dim}} \subset \mathrm{UI} \subset \mathrm{FD}$
of inclusions.
There are examples of groups that satisfy (UI) but are not finite dimensional.
Consider the abelian subgroup
\[ {\mathcal G} := \{ (x, y+ f(x)) : f \in {\mathbb C} \{ x \}, \ f(0) = 0 \} \]
of $\diff{}{2}$.
Consider a finitely determined subgroup $G$  of ${\mathcal G}$
that is non-finite dimensional. For instance we can consider the group generated by
the diffeomorphisms $\phi_{j}(x,y) = (x, y + d_{j} x + x^{j+1})$ for $j \in {\mathbb N}$.
If the set $\{ d_1, d_2, \cdots \}$ is linearly independent over ${\mathbb Q}$ then
the first jet of an element $\phi$ of $G$ determines $\phi$.
Moreover $G$ is infinitely dimensional since the polynomials of the form
$d_{j} x + x^{j+1}$ for $j \in {\mathbb N}$ are linearly independent over ${\mathbb C}$.
Anyway $G$ satisfies (UI). A complete proof is provided in Proposition \ref{pro:UIn} but
in order to illustrate how finite determination
implies (UI) in this context, let us consider the intersection multiplicity
\begin{equation}
\label{equ:sexa}
(\phi  \{ y =g(x) \},  \{ y =g(x) \})
\end{equation}
for a smooth curve  $y = g(x)$.
Given $\phi (x,y) = (x, y+ f(x))$ we have
\[ \phi \{ y =g(x) \} = \{ y = f(x) + g(x) \}. \]
Thus Expression (\ref{equ:sexa}) is equal to
the vanishing order of $f(x)$ at $x=0$.
Since it  is bounded  by hypothesis,
so is the intersection multiplicity (\ref{equ:sexa})
whenever it is different than $\infty$.
Analogously we have ${\mathrm F}_{\dim} \subsetneq \mathrm{UI}$ (see sequence (\ref{equ:inc}))
in every dimension $n \geq 2$.

What happens with the inclusion $\mathrm{UI} \subset \mathrm{FD}$? It is not difficult to find examples
of subgroups of $\diff{}{n}$ that satisfy (FD) but not (UI) in dimension $n \geq 3$ (Lemma \ref{lem:exa}).
The main result of this paper is the following theorem:
\begin{main}
\label{teo:main}
Let $G$ be a subgroup of $\diffh{}{2}$. Then $G$ satisfies the uniform intersection property
if and only if $G$ satisfies the finite determination property.
\end{main}

This theorem characterizes the groups of diffeomorphisms satisfying the uniform intersection
property in the two dimensional case. It unifies and extends the previous results in the literature.
Such results showed that the following types of groups hold the uniform intersection property:
cyclic groups \cite{Arnold:intersection}, finitely generated abelian groups \cite{Seigal-Yakovenko:ldi},
Lie groups (with finitely many connected components) \cite{Binyamini:finite}, finitely generated
virtually nilpotent groups and virtually polycyclic groups \cite{JR:finite}.
The finite determination condition is the culmination of the journey.
It is much simpler than the previous best sufficient conditions to guarantee the uniform
intersection property, namely $\mathrm{Lie}$ and $\mathrm{F_{\dim}}$ (Equation (\ref{equ:inc})).
In order to show the (UI) property, finite determination is the simplest necessary condition.
Surprisingly, it is also a sufficient condition.

\strut

We can characterize finitely determined subgroups of $\diffh{}{2}$ in terms of their actions on
the space of curves.
\begin{teo}
\label{teo:fd_discrete}
Let $G$ be a subgroup of $\diffh{}{2}$.  Then $G$ is finitely determined if and only if
the action of $G$ on the topological space of formal irreducible curves has discrete orbits.
\end{teo}
We say that a formal irreducible curve $\gamma$ (cf. Definition \ref{def:curve})
satisfies the property $\mathrm{(UI)}_{\gamma}$
if its orbit is discrete (see Definition \ref{def:topology} for a description of the topology on the space of
curves).  Anyway, let us remark that  the $G$-orbit of $\gamma$ is discrete if and only if
the function $\phi \mapsto (\phi (\gamma), \gamma)$ (of $G$) takes finitely many values
(Lemma \ref{lem:redc}).

 \strut

The Main Theorem holds true in the $C^{\infty}$ setting where we consider subgroups of the group
$\mathrm{Diff}^{\infty} ({\mathbb R}^{2},0)$ of $C^{\infty}$ local diffeomorphisms (defined in  a neighborhood of the origin of ${\mathbb R}^{2}$) and germs of $C^{\infty}$ subvarieties in the definition of the uniform intersection property.
The reason is that we fall in the setting of the Main Theorem  by considering the Taylor series expansion at the origin of diffeomorphisms and subvarieties.


The implication (FD) $\implies$ (UI) is the
difficult part of the proof of the Main Theorem.
As a first reduction we show that
the group $G$ satisfies (UI) if and only if $\mathrm{(UI)}_{\gamma}$ holds for any
irreducible curve $\gamma$ (Proposition  \ref{pro:ui_discrete} ).

Given a pair $(G,\gamma)$ consisting of a finitely determined subgroup of $\diffh{}{2}$
and a formal irreducible curve $\gamma$,
we say that $(J, \gamma)$ is a reduced pair if $J$ is a subgroup of $G$ and
$G$ satisfies $\mathrm{(UI)}_{\gamma}$ if and only
if $J$ does. Fix an irreducible curve $\gamma$. It is possible to show that
there exists a reduced pair $(J,\gamma)$ such that one of the next situations happens:
\begin{itemize}
\item $\gamma$ is $J$-invariant or
\item $J$ is finite dimensional or
\item $J$ is abelian.
\end{itemize}
The first case is trivial since the function of $J$ defined by
$\phi \mapsto (\phi (\gamma), \gamma)$ is constant and equal to
$\infty$. In the second case $J$ satisfies (UI) as
a consequence of the inclusion $F_{\dim} \subset \mathrm{UI}$ and more
precisely of Theorem 1.5 of \cite{JR:finite}.
It remains to consider the case where $J$ is abelian and infinitely dimensional.
Then $J$ is a subgroup of the exponential
$\mathrm{exp} ({\mathfrak g})$ of an infinitely dimensional abelian complex Lie algebra
${\mathfrak g}$ of formal vector fields.
The property $\dim_{\mathbb C} {\mathfrak g} = \infty$ forces ${\mathfrak g}$ to be very special.
Indeed there exists a formal vector field $X$ such that
every element of ${\mathfrak g}$ is of the form $f X$ where $f \in {\mathbb C}[[x,y]]$ and $X(f)=0$
(Proposition \ref{pro:nla}).
This situation was described above for $X  = \frac{\partial }{\partial y}$.
The proof of property (UI) can be reduced to the proof of such a case.

We consider several kind of reductions $(J, \gamma)$ of a pair $(G,\gamma)$.
One reduction consists in replacing $G$ with a finite index or suitable normal subgroup $J$.
We can also reduce the pair by replacing
$G$ with the subgroup $J$ of elements of $G$ whose linear parts preserve the tangent direction
$\ell$ of $\gamma$ at the origin.
Furthermore this reduction can be generalized by considering iterated tangents
(or infinitely near points) of the curve $\gamma$.
It is then natural to use desingularization techniques (of curves and foliations)
in the proof of the Main Theorem
to obtain simpler expressions for the pair $(G,\gamma)$.

\section{Notations}
Let us introduce some notations that will be useful in the paper.
\subsection{Formal power series and curves}
\begin{defi}
We denote by ${\mathcal O}_{n}$ (resp. $\hat{\mathcal O}_{n}$)
the local ring ${\mathbb C}\{ z_1, \cdots, z_n \}$  (resp. ${\mathbb C}[[ z_1, \cdots, z_n ]]$)
of convergent (resp. formal) power series with complex coefficients in $n$ variables.
We denote by ${\mathfrak m}$ the maximal ideal of  $\hat{\mathcal O}_{n}$.
We define $\hat{K}_n$ as the field of fractions of $\hat{\mathcal O}_{n}$.
\end{defi}
\begin{defi}
Let $f \in \hat{\mathcal O}_n$. We define the $k$-jet $j^{k} f$ as the polynomial
$P_{k} (z_1, \cdots, z_n)$ of degree less or equal than $k$ such that $f - P_k \in {\mathfrak m}^{k+1}$.
\end{defi}
Let us define formal curves for an ambient space of dimension $2$.
\begin{defi}
\label{def:curve}
A formal curve is an ideal $(f)$ of $\hat{\mathcal O}_{2}$ where $f$ belongs to the maximal ideal of
$\hat{\mathcal O}_{2}$. We say that the curve is irreducible if $f$ is an irreducible element of
$\hat{\mathcal O}_{2}$.
\end{defi}
\subsection{Formal diffeomorphisms and vector fields}
\begin{defi}
We define the group of formal diffeomorphisms in $n$ variables as
\[ \diffh{}{n} = \{ (\phi_1, \cdots, \phi_n) \in {\mathfrak m}^{n} :
(j^{1} \phi_1, \cdots, j^{1} \phi_n) \in \mathrm{GL}(n, {\mathbb C}) \}. \]
The group operation is defined in such a way that the composition
\[ (\rho_1, \cdots, \rho_n) = (\phi_1, \cdots, \phi_n) \circ (\eta_1, \cdots, \eta_n) \]
satisfies $j^{k} \rho_j = j^{k} (j^{k} \phi_{j} \circ (j^{k} \eta_1, \cdots, j^{k} \eta_n))$
for all $1 \leq j \leq n$ and $k \in {\mathbb N}$.
The subgroup $\diff{}{n} = \diffh{}{n} \cap {\mathcal O}_{n}^{n}$ is called
{\it group of local biholomorphisms} in $n$ variables.
\end{defi}
\begin{rem}
The group $\diff{}{n}$ consists of germs of biholomorphisms by the inverse function theorem.
\end{rem}
\begin{defi}
Given $\phi = (\phi_1, \cdots, \phi_n) \in \diffh{}{n}$ we denote by
$j^{1} \phi$ the linear part at the origin $(j^{1} \phi_1, \cdots, j^{1} \phi_n)$ of $\phi$.
Given a subgroup $G$ of $\diffh{}{n}$,
we define $j^{1} G$ as the group $\{ j^{1} \phi : \phi \in G \}$ of linear parts of elements of $G$.
\end{defi}
\begin{defi}
We denote by $\diffh{u}{n}$ the set of formal diffeomorphisms consisting of elements
$(\phi_1, \cdots, \phi_n)$ such that $j^{1} \phi$ is a a unipotent linear isomorphism,
i.e. $1$ is its unique eigenvalue.
We define the group of {\it tangent to the identity} formal diffeomorphisms as
\[ \diffh{1}{n} = \{ \phi = (\phi_1, \cdots, \phi_n) \in \diffh{}{n} : j^{1} \phi = \mathrm{Id} \},  \]
i.e. is the subgroup of $\diffh{}{n}$ of elements with identity linear part.
\end{defi}
\begin{defi}
We define the Lie algebra of formal vector fields in $n$ variables as
\[ \hat{\mathfrak X} \cn{n} =
\left\{ \sum_{j=1}^{n}  f_{j} (z_1, \cdots, z_n) \frac{\partial}{\partial z_j}  :
 f_1, \cdots, f_n \in {\mathfrak m} \right\}. \]
It can be interpreted as a derivation of the ${\mathbb C}$-algebra ${\mathfrak m}$.
If an element of $\hat{\mathfrak X} \cn{n}$ satisfies $f_1, \cdots, f_n \in {\mathcal O}_n$ we say that it
is a (singular) holomorphic local vector field.
We denote by ${\mathfrak X} \cn{n}$ the set of holomorphic singular local vector fields.
\end{defi}
\begin{defi}
Given $X = \sum_{j=1}^{n}  f_{j} (z_1, \cdots, z_n) \frac{\partial}{\partial z_j}  \in  \hat{\mathfrak X} \cn{n}$,
we say that $X$ is {\it nilpotent} if
$j^{1} X :=  \sum_{j=1}^{n}  j^{1} f_{j} (z_1, \cdots, z_n) \frac{\partial}{\partial z_j}$
is a linear nilpotent vector field. Denote by $\hat{\mathfrak X}_{N} \cn{n}$ the set of formal nilpotent vector
fields.
\end{defi}
\subsection{The central and the derived series}
\begin{defi}
\label{def:solv}
Let $H,L$ be subgroups of $G$. We define
$[H,L]$ as the group generated by the commutators $[h,l]:= h l h^{-1} l^{-1}$ for $h \in H$, $l \in L$.
We define the derived series
\[ G^{(0)} = G, \ G^{(1)} = [G^{(0)}, G^{(0)}],  \cdots, \ G^{(k+1)} = [G^{(k)}, G^{(k)}],  \cdots \]
of the group $G$. The group $G^{(k)}$ is called the $k$th derived group of $G$.
The derived group $G^{(1)}$ is also denoted by $G'$.

We define by $\ell (G) = \min (\{ k \geq 0 : G^{(k)} = \{\mathrm{Id} \} \} \cup \{ \infty \})$ the {\it derived length} of $G$.
We say that $G$ is {\it solvable} if $\ell (G) < \infty$.
\end{defi}
\begin{defi}
\label{def:nilp}
 We define the lower central series
\[ {\mathcal C}^{0} G = G, \ {\mathcal C}^{1} G = [{\mathcal C}^{0} G, G],  \cdots,
\ {\mathcal C}^{k+1} G = [{\mathcal C}^{k} G, G],  \cdots \]
of the group $G$. We say that $G$ is {\it nilpotent if} there exists $k \geq 0$ such that
$ {\mathcal C}^{k} G =  \{\mathrm{Id} \}$.
Moreover $\min (\{ k \geq 0 : {\mathcal C}^{k} G = \{\mathrm{Id} \} \} \cup \{ \infty \})$
is called the {\it nilpotency class} of $G$.
\end{defi}
\begin{defi}
Given Lie subalgebras ${\mathfrak h}$, ${\mathfrak l}$ of a Lie algebra ${\mathfrak g}$.
We define $[ {\mathfrak h}, {\mathfrak l}]$ as the Lie algebra generated by the Lie brackets
$[X,Y]$ where $X \in {\mathfrak h}$ and $Y \in {\mathfrak l}$.
We can define the derived Lie algebra ${\mathfrak g}'$, the derived series
$({\mathfrak g}^{(k)})_{k \geq 0}$, the central lower series $({\mathcal C}^{k} {\mathfrak g})_{k \geq 0}$,
nilpotent and solvable Lie algebras analogously as in Definitions \ref{def:solv} and \ref{def:nilp}.
\end{defi}
\section{Finite determination and uniform intersection}
In this section we introduce the first connections between the finite determination property and
the uniform intersection property. In particular we show that the latter property implies the former
in dimension $n \geq 2$ but the reciprocal is not true for any $n \geq 3$.
In particular the analogue of the Main Theorem for $n \geq 3$ does not hold true.
\begin{defi}
Given $k \in {\mathbb N}$ we say that a subgroup $G$ of $\diffh{}{n}$ is $k$-{\it finitely determined}
if $\{ \phi \in G : j^{k} \phi = \mathrm{Id} \} = \{ \mathrm{Id} \}$. We say that $G$ has the {\it finite determination property} (FD) if
it is $k$-finitely determined for some $k \in {\mathbb N}$.
\end{defi}
\begin{rem}
The elements of a $k$-finitely determined subgroup of $\diffh{}{n}$ are determined by their $k$-jets.
\end{rem}
\begin{defi}
\label{def:intmul}
The intersection multiplicity of ideals $I$ and $J$ of $\hat{\mathcal O}_{n}$ is defined by
the dimension of the complex vector space $ \hat{\mathcal O}_{n} / (I,J) $.
\end{defi}
\begin{rem}
Formal schemes are given by ideals of the ring $\hat{\mathcal O}_{n}$.
The previous definition provides un upper bound for the
usual definition of intersection multiplicity of formal schemes (or ideals) with their associated cycle structure
(cf. \cite[Proposition 8.2]{Fulton:inter}).
Moreover both definitions coincide if one of them is equal to $\infty$.
Thus the uniform intersection property with respect to Definition \ref{def:intmul} implies the (UI)
for the usual definition of the intersection multiplicity.
\end{rem}
\begin{defi}
We say that a subgroup $G$ of $\diffh{}{n}$ has the {\it uniform intersection property} (UI) if
for any pair $\alpha$, $\beta$ of formal schemes, the subset
$\{ (\phi (\alpha), \beta) : \phi \in G \}$ of ${\mathbb N} \cup \{\infty\}$ is finite.
Equivalently given ideals $I, J$ of $\hat{\mathcal O}_{n}$ the set
$\{ (\phi^{*} (I), J) : \phi \in G \}$ is finite.
%
In other words the intersection multiplicities between the shifts of $\alpha$ by the action
of the group and $\beta$ (different from infinity) are uniformly bounded.
\end{defi}
\begin{lem}
\label{lem:easy}
Let $G$ be a subgroup of $\diffh{}{n}$ ($n \geq 2$) that satisfies (UI).
Then it also satisfies (FD).
\end{lem}
\begin{proof}
Let us consider first the case $n=2$. Suppose that $G$ does not satisfy the finite determination property.
There exists a sequence $(\phi_{k})_{k \geq 1}$ of elements in $G \setminus \{\mathrm{Id}\}$
such that $j^{k} \phi_k = \mathrm{Id}$ for any $k \in {\mathbb N}$.
Given any formal irreducible curve $\gamma$
we have $\lim_{j \to \infty} (\phi_{j}(\gamma), \gamma)  = \infty$.
Property (UI) implies $\phi_{j} (\gamma) = \gamma$ for any $j \geq j_0$
and some $j_{0} \in {\mathbb N}$.

Consider power series $f, g \in {\mathbb C} \{x,y\}$ such that $f(0,0)=g(0,0)=0$ and
$f$ and $g$ have no common irreducible factors. Given any $c \in {\mathbb C}$ the curve
$\gamma_{c}$ given by the equation $f(x,y) - c g(x,y)=0$
(or equivalently the ideal $(f(x,y) - c g(x,y))$)
is invariant by $\phi_{j}$ for $j>>1$.
We define
\[ A_{l} = \{ c \in {\mathbb C} : \phi_{j} (\gamma_{c}) = \gamma_{c} \ \forall j \geq l \} . \]
Since $\cup_{l \in {\mathbb N}} A_{l} = {\mathbb C}$, it follows that $A_{l_0}$ is uncountable for some
$l_0 \in {\mathbb N}$.
Since $(f \circ \phi_j) g - f (g \circ \phi_j)$ vanishes in an uncountable many curves,
we obtain $(f \circ \phi_j) g - f (g \circ \phi_j)=0$ and then
$\frac{f}{g} \circ \phi_j \equiv \frac{f}{g}$ for any $j \geq l_0$.
By applying the previous property to $y/x$, $y^2/x$ and $y/x^{2}$ we obtain
\[ x \circ \phi_{j} \equiv x, \ \ y \circ \phi_j \equiv y \implies \phi_j \equiv \mathrm{Id} \]
for any $j >>1$ contradicting that $(\phi_{j})_{j \geq 1}$ is contained in $G \setminus \{\mathrm{Id}\}$.

Consider a general $n \geq 2$. Let $N$ be a smooth manifold of codimension $1$.
Given a curve $\gamma$ contained in $N$, we have $\phi_{j}(\gamma) \subset N$ for $j>>1$
analogously as for the case $n=2$. Therefore we obtain $\phi_{j}(N) = N$ for any $j >>1$.
Any manifold $M$ is an  intersection of hypersurfaces and hence
$\phi_j (M) = M$ for $j>>1$. The proof of the case $n=2$ implies then $(\phi_{j})_{|M} \equiv \mathrm{Id}$
for any smooth manifold $M$ of dimension $2$ and $j>>1$. We deduce $\phi_{j} \equiv \mathrm{Id}$ for
$j>>1$, obtaining a contradiction.
\end{proof}
\begin{lem}
\label{lem:exa}
Let $n \geq 3$. There exists a subgroup $G$ of $\diff{}{n}$ that satisfies (FD) but
does not hold (UI).
\end{lem}
\begin{proof}
We denote
\[ \phi_j (x_1,x_2, x_3, \cdots, x_n) = (x_1, x_2 + d_j x_{1}^{2} + x_{3}^{j+2} , x_{3}, \cdots, x_n) \in \diff{}{n} \]
for any $j \in {\mathbb N}$ where $\{ d_1, d_2, \cdots \}$ is linearly independent over ${\mathbb Q}$.
We define the group $G= \langle \phi_1, \phi_2, \cdots \rangle$. It is clearly abelian.
The condition of linear independence of $\{ d_1, d_2, \cdots \}$ implies that any diffeomorphism $\phi \in G$
that is not equal to the identity map, has a non-vanishing second jet. Therefore $G$ satisfies (FD).

Denote $\alpha = \{ x_1=x_2=0\}$ and $\beta = \{x_2= x_{4} = \cdots = x_n =0\}$. We have
\[ (\phi_{j}^{-1} (\alpha), \beta) =
\dim_{\mathbb C} \frac{{\mathcal O}_{n}}{(x_1, x_2+ d_j x_{1}^{2} + x_{3}^{j+2},x_2, x_4, \cdots, x_n)} \]
\[ = \dim_{\mathbb C}
\frac{{\mathcal O}_{n}}{(x_1, x_2, x_{3}^{j+2}, x_4, \cdots, x_n)} = j+2  \]
for any $j \in {\mathbb N}$.  Thus $G$ does not satisfy (UI).
\end{proof}
\begin{rem}
The group $G$ defined in Lemma \ref{lem:exa} is not finitely generated. This is fundamental in the example
since finitely generated abelian subgroups of $\diffh{}{n}$ satisfy (UI) \cite{Seigal-Yakovenko:ldi}.
\end{rem}
\section{Intersection properties of curves}
In this section we will see that
the (UI) condition in dimension $2$ is a property about orbits of curves.
We will introduce some notations and techniques, mainly related to blow-ups, that allow
to reduce the proof of the Main Theorem to more manageable cases.

Let $G$ be a subgroup of $\diffh{}{2}$.
Given ideals $I$ and $J$ the set $\{ (\phi^{*} (I), J) : \phi \in G \}$ is bounded
if $\sqrt{I}$ (or $\sqrt{J}$) contains
the maximal ideal $(x,y)$ since then $(x,y)^{k} \subset I$ for some $k \geq 1$ and it is clear that
\[ \dim_{\mathbb C} \frac{\hat{\mathcal O}_{2}}{(\phi^{*}(I),J)} \leq
\dim_{\mathbb C} \frac{\hat{\mathcal O}_{2}}{I} \leq
 \dim_{\mathbb C} \frac{\hat{\mathcal O}_{2}}{(x,y)^{k}}  = k(k-1)/2 < \infty  \]
for any $\phi \in G$. The case $I=0$ or $J=0$ is also simple since
$\dim_{\mathbb C} \hat{\mathcal O}_{2}/ (\phi^{*}(I),J)$ is constant as a function of $\phi \in G$.
Hence we can suppose $\sqrt{I} = (f)$, $\sqrt{J}=(g)$ where $f= f_1 \cdots f_j$, $g =g_1 \cdots g_l$
and $f_1, \cdots, f_j$ (resp. $g_1, \cdots, g_l$) are pairwise relatively prime irreducible elements of
$\hat{\mathcal O}_{2}$. There exists $k \in {\mathbb N}$
such that $(f^{k}) \subset I \subset (f)$ and $(g^{k}) \subset J \subset (g)$.
Given $\phi \in \diffh{}{2}$ we obtain
\[ \dim_{\mathbb C} \frac{\hat{\mathcal O}_{2}}{(\phi^{*}(I),J)} \leq
\dim_{\mathbb C} \frac{\hat{\mathcal O}_{2}}{( f^{k} \circ \phi ,g^{k})}  \leq k^{2}
\dim_{\mathbb C} \frac{\hat{\mathcal O}_{2}}{( f \circ \phi ,g)} = \]
\[ k^{2} \sum_{1 \leq a \leq j, \ 1 \leq b \leq l} \dim_{\mathbb C} \frac{\hat{\mathcal O}_{2}}{( f_a \circ \phi ,g_b)}=
k^{2} \sum_{1 \leq a \leq j, \ 1 \leq b \leq l} (\phi^{-1} (\alpha_a), \beta_b) \]
where $\alpha_a$ (resp. $\beta_b$) is the formal irreducible curve of ideal $(f_a)$ (resp. $(g_b)$).
The previous discussion leads to the first reduction of the setting of the problem.
\begin{rem}
\label{rem:redfic}
Let $G$ be a subgroup of $\diffh{}{2}$.
Then $G$ satisfies (UI) if and only if given any pair of formal
irreducible curves $\alpha$ and $\beta$ the set
$\{ (\phi (\alpha), \beta) : \phi \in G \}$ is finite.
\end{rem}
\subsection{The action of a group on the space of curves}
\label{sec:action}
So far we saw that it suffices to consider curves to check the uniform intersection property for a
group $G$. The main result in this subsection is Proposition \ref{pro:ui_discrete}
that provides a stronger property: $G$ satisfies (UI) if and only if all the orbits of formal irreducible
curves are discrete. We will define the natural topology in the space of curves. Anyway,
the orbit of a formal irreducible curve $\gamma$ is discrete if and only if
$\{ (\phi (\gamma), \gamma) : \phi \in G \}$ is finite.

\strut

Let $\gamma$ be a formal irreducible curve given by an ideal $(f)$ where $f$
is an irreducible element of
$\hat{\mathcal O}_{2}$.  Denote $\gamma_0 = \gamma$ and $p_0 = (0,0) \in {\mathbb C}^{2}$.
\begin{defi}
Let $f \in \hat{\mathcal O}_{2}$. We denote by $m_{0} (f)$ the multiplicity of $f$ at the origin. It is
the integer number $m \geq 0$ such that $f \in (x,y)^{m} \setminus  (x,y)^{m+1}$.
Given a formal irreducible curve $\gamma$ given by an ideal $(f)$ we define the
multiplicity of $\gamma$ at the origin
by $m_{0} (\gamma) = m_{0}(f)$. We say that $\gamma$ is smooth if $m_{0}(\gamma)=1$.
\end{defi}
\begin{defi}
Let $f$ be an irreducible element of $\hat{\mathcal O}_2$ defining a formal irreducible curve $\gamma$.
There exists $(a,b) \in {\mathbb C}^{2} \setminus \{(0,0)\}$ such that
\[ f  - (ax + by)^{m_{0}(\gamma)} \in (x,y)^{m_{0}(\gamma)+1}, \]
cf. \cite[Corollary 2.2.6, p. 46]{Casas:sing}.
The line $\ell = \{ ax + by =0\}$ is the {\it tangent direction} of $\gamma$ at the origin.
\end{defi}
Consider the blow-up $\pi_1: \widetilde{{\mathbb C}^{2}} \to {\mathbb C}^{2}$ of the origin of
${\mathbb C}^{2}$.
Let $p_1 (\gamma)$ be the first infinitely near point of $\gamma$, i.e.
the point in $\pi_{1}^{-1}(p_0)$ corresponding to the tangent direction
$\ell$ of $\gamma$ at the origin. Denote by $\gamma_1$ the strict transform of $\gamma$.
More precisely, suppose $\ell = \{ y=0\}$ and consider  coordinates $(x,t)$
in the first chart of the blow-up
of the origin. Then $\pi_{1}(x,t)=(x,xt)$ is the expression of $\pi_{1}$ in local coordinates.
The formal irreducible curve $\gamma_1$ is equal to the prime ideal $(f(x,xt)/x^{m_{0}(\gamma)})$
of ${\mathbb C}[[x,t]]$. The point $p_1 (\gamma)$ has coordinates $(x,t)=(0,0)$.
We denote ${\mathcal X}_0 = {\mathbb C}^{2}$ and ${\mathcal X}_1 = \widetilde{{\mathbb C}^{2}}$.
By repeating this process we obtain a sequence of blow-ups
\[ \hdots \stackrel{\pi_{m+2}}{\longrightarrow}
{\mathcal X}_{m+1} \stackrel{\pi_{m+1}}{\longrightarrow} {\mathcal X}_{m} \stackrel{\pi_m}{\longrightarrow}
\hdots \stackrel{\pi_3}{\longrightarrow} {\mathcal X}_{2} \stackrel{\pi_2}{\longrightarrow}  {\mathcal X}_{1}
\stackrel{\pi_1}{\longrightarrow} {\mathcal X}_{0} \]
where $p_0=(0,0)$, $\gamma_0 = \gamma$, $\pi_{k+1}$ is the blow-up of ${\mathcal X}_k$
at $p_k (\gamma)$ and $p_{k+1} (\gamma)$ is the point in $\pi_{k+1}^{-1} (p_k (\gamma))$
 corresponding to the tangent direction of $\gamma_k$ at $p_k (\gamma)$
for $k \geq 0$. The strict transform $\gamma_{k+1}$ of $\gamma_k$ by the blow-up $\pi_{k+1}$ passes through the
point $p_{k+1} (\gamma)$ for any $k \geq 0$.
Then  $(p_{k} (\gamma))_{k \geq 0}$ and $(\gamma_k)_{k \geq 0}$ are the sequences
of infinitely near points and strict transforms of $\gamma$ respectively.
We denote $p_k = p_k(\gamma)$ if $\gamma$ is implicit.

The previous construction can be used to provide a natural topology in the space of curves.
\begin{defi}
\label{def:topology}
Given a finite sequence  ${\mathfrak s}=(p_{k})_{0 \leq k \leq n}$ ($n \geq 0$) of
infinitely near points
we define the set ${\mathcal U}_{\mathfrak s}$ consisting of the formal irreducible curves
$\gamma$ such that $p_k (\gamma)=p_k$ for any $0 \leq k \leq n$.
The sets of the form ${\mathcal U}_{\mathfrak s}$ form the base of a topology on the set
of formal irreducible curves.
\end{defi}
\begin{rem}
The space of formal irreducible curves can be interpreted as a subset of the valuations
of ${\mathbb C}[[x,y]]$ that take values in ${\mathbb R}^{+} \cup \{\infty\}$. Such space is a
tree and we just considered the topology induced on the space of curves by any of the
natural topologies of the valuative tree (weak, strong, thin, Hausdorff-Zariski)  \cite{Fa-Jo:LN}.
\end{rem}
\begin{defi}
Given $k \geq 1$
we denote by $m_{k}(\gamma)$ the multiplicity of $\gamma_k$ at $p_k$.
\end{defi}
\begin{lem}
\label{lem:redc}
Let $G$ be a subgroup of $\diffh{}{2}$.
Let $(\alpha_n)_{n \geq 1}$, $(\beta_{n})_{n \geq 1}$ be $G$-orbits of formal irreducible curves.
Let $\gamma$ be a formal irreducible curve.
Then
\begin{itemize}
\item $\lim_{n \to \infty} \alpha_n = \gamma$ if and only if
 $\lim_{n \to \infty} (\alpha_n, \gamma)=\infty$.
\item $\lim_{n \to \infty} (\alpha_n, \gamma) = \lim_{n \to \infty} (\beta_n, \gamma)= \infty$ implies
$\lim_{n \to \infty} (\alpha_n, \beta_n)=\infty$.
\end{itemize}
\end{lem}
\begin{proof}
Given formal irreducible curves $\alpha$, $\beta$ we have
\begin{equation}
\label{equ:inp}
 (\alpha, \beta) = m_{0} (\alpha) m_{0} (\beta)  + \sum_{k=1}^{\infty} m_{k} (\alpha) m_{k} (\beta) \delta_k
\end{equation}
where $\delta_k=1$ if the first $k$ infinitely near points of $\alpha$ and $\beta$ coincide
and $\delta_k=0$ otherwise
(cf. \cite[Corollary 8.30]{Ilya-Yako}).
The property  $\lim_{n \to \infty} \alpha_n = \gamma$ implies
$\lim_{n \to \infty} (\alpha_n, \gamma)=\infty$ by Equation (\ref{equ:inp}).

The sequences $(m_{k} (\alpha_n))_{k \geq 0}$  are decreasing
and do not depend on $n$. Hence the property $\lim_{n \to \infty} (\alpha_n, \gamma) =\infty$ implies
$\lim_{n \to \infty} \alpha_n = \gamma$.

Suppose $\lim_{n \to \infty} (\alpha_n, \gamma) = \lim_{n \to \infty} (\beta_n, \gamma)= \infty$.
By the first part of the proof $\alpha_n$ and $\beta_n$
share the first $a_n$ infinitely near points where
$\lim_{n \to \infty} a_n = \infty$.  As a consequence we obtain
$\lim_{n \to \infty} (\alpha_n, \beta_n) = \infty$ by Equation (\ref{equ:inp}).
\end{proof}
\begin{rem}
\label{rem:dis_closed}
A subgroup $G$ of  $\diffh{}{2}$ satisfies the
uniform intersection property if and only if every $G$-orbit
of formal irreducible curves is discrete and closed by Lemma \ref{lem:redc}
\end{rem}

We can understand the uniform intersection property as a property of orbits of curves.
\begin{pro}
\label{pro:closure}
Let $G$ be a subgroup of $\diffh{}{2}$. Let ${\mathcal O}=(\phi (\gamma))_{\phi \in G}$
be a $G$-orbit of formal irreducible curves. Then
\begin{itemize}
\item ${\mathcal O}$ is discrete and hence it is closed or
\item $\overline{\mathcal O}$ is a Cantor set, meaning it is a closed set with no isolated points and
has the cardinality of the continuum.
\end{itemize}
\end{pro}
\begin{proof}
Suppose that ${\mathcal O}$ is discrete. Let $\alpha \in \overline{\mathcal O}$. There exists
a sequence $(\phi_{j})_{j \geq 1}$ in $G$
such that $\lim_{j \to \infty} \phi_j (\gamma) = \alpha$.
We get $\lim_{j \to \infty} ( \phi_j (\gamma) , \alpha)= \infty$ by the first property
in Lemma \ref{lem:redc}. Moreover, the second property of Lemma \ref{lem:redc} implies
$\lim_{j \to \infty} (\phi_j(\gamma), \phi_{j+1}(\gamma)) = \infty$.
We deduce $\lim_{j \to \infty} ((\phi_{j+1}^{-1} \circ \phi_{j})(\gamma), \gamma) = \infty$
and then $\lim_{j \to \infty} (\phi_{j+1}^{-1} \circ \phi_{j})(\gamma) = \gamma$.
Since ${\mathcal O}$ is discrete, we obtain
$\phi_{j}(\gamma)  = \phi_{j+1} (\gamma)$ for $j \in {\mathbb N}$ sufficiently big.
In particular $\alpha$ belongs to ${\mathcal O}$ for any $\alpha \in  \overline{\mathcal O}$
and hence ${\mathcal O}$ is closed.

 Suppose that ${\mathcal O}$ is non-discrete from now on.
 There exists a sequence $(\phi_{j})_{j \geq 1}$ in $G$
such that $\lim_{j \to \infty} \phi_j (\gamma) = \gamma$ and $\phi_j (\gamma) \neq \gamma$ for any
$j \geq 1$.
Let $a_j$ be the index of the first infinitely near point of $\gamma$ that is not an infinitely
near point of $\phi_j (\gamma)$, i.e. $p_k (\gamma)= p_k (\phi_j (\gamma))$ if
$0 \leq k < a_j$ and $p_{a_j} (\gamma) \neq  p_{a_j} (\phi_j (\gamma))$.
The construction of $(\phi_{j})_{j \geq 1}$ implies $\lim_{j \to \infty} a_j=\infty$.
We can suppose that $(a_j)_{j \geq 1}$ is a strictly increasing sequence.

Let $\alpha \in \overline{\mathcal O}$. It is of the form $\lim_{n \to \infty} \gamma_n$ where
$(\gamma_n)_{n \geq 1}$ is a sequence of ${\mathcal O}$.
Fix $j \geq 1$. There exists $n_0 \in {\mathbb N}$ such that
$p_k (\gamma_n) = p_k (\alpha)$ for all $n \geq n_0$ and $0 \leq k \leq a_j$.
Let $\psi$ be an element of $G$ such that $\psi (\gamma) = \gamma_{n_0}$.
We define $\alpha_{j,k} = (\psi \circ \phi_j \circ \psi^{-1}) (\gamma_k)$ for $k \geq 1$.
By construction the limit $\alpha_j := \lim_{k \to \infty} \alpha_{j,k}$ exists and belongs to
$\overline{\mathcal O}$. Again by construction, we have
$p_k (\alpha_j)=p_k (\alpha)$ for any $0 \leq k < a_j$ and
$p_{a_j} (\alpha_j) \neq p_{a_j} (\alpha)$. We deduce that $\alpha$ is not an isolated point of
$\overline{\mathcal O}$.

Sequence of infinitely near points can be identified with sequences of points in the
complex projective space ${\bf CP}^{1}$. Thus such set has the cardinality of the continuum as
it does the set of formal irreducible curves. It suffices to show that $\overline{\mathcal O}$
has at least the cardinality of the continuum. Indeed there is an injective map from the set of
sequences in the set $\{0,1\}$ to $\overline{\mathcal O}$. By the previous paragraph there
exist $\gamma_0, \gamma_1 \in {\mathcal O}$ such that
$p_k (\gamma_0)=p_k (\gamma_1)$ for any $0 \leq k < a_1$ and
$p_{a_1} (\gamma_0) \neq p_{a_1} (\gamma_1)$.
Then we can construct $\gamma_{l,s}$ ($l,s \in \{0,1\}$) such that
\begin{itemize}
\item $p_k(\gamma_{l,s})=p_k (\gamma_l)$ for all $0 \leq k \leq a_1$, $l,s \in \{0,1\}$.
\item $p_k(\gamma_{l,0}) = p_k (\gamma_{l,1})$ for all $0 \leq k < a_2$ and $l \in \{0,1\}$.
\item $p_{a_2}(\gamma_{l,0}) \neq p_{a_2} (\gamma_{l,1})$ for any $l \in \{0,1\}$.
\end{itemize}
Analogously, given $\gamma_{l_1, \hdots, l_r}$ $l_1, \hdots, l_r \in \{0,1\}$ we build
$\gamma_{l_1, \hdots, l_r,0}$ and $\gamma_{l_1, \hdots, l_r,1}$ such that they share the
same infinitely near points until the level $a_{r+1}$ (non-included) and they have the
same infinitely near points as $\gamma_{l_1, \hdots, l_r}$ until level $a_r$ (included).
Given any sequence $(l_k)_{k \geq 1}$ in $\{0,1\}$ the limit
$\lim_{k \to \infty } \gamma_{l_1, \hdots, l_k}$ exists by construction and limits corresponding
to different sequences are different.
Since the set of sequences in $\{0,1\}$ has the cardinality of the continuum,
so it does $\overline{\mathcal O}$.
\end{proof}
\begin{cor}
Let $G$ be a countable subgroup of $\diffh{}{2}$.
Then a $G$-orbit of formal irreducible curves is closed if and only if it is discrete.
\end{cor}
\begin{defi}
Let $G$ be a subgroup of $\diffh{}{2}$ and $\gamma$ a formal curve. We say that $G$ satisfies the
property $(\mathrm{UI})_{\gamma}$ if the $G$-orbit of $\gamma$ is discrete, i.e. if
$\{ (\phi (\gamma), \gamma) : \phi \in G \}$ is finite.
\end{defi}
The following result is a corollary of Remark \ref{rem:dis_closed} and Proposition
\ref{pro:closure}.
\begin{pro}
\label{pro:ui_discrete}
Let $G$ be a subgroup of $\diffh{}{2}$.  Then $G$ satisfies the uniform intersection property
if and only if every $G$-orbit of formal irreducible curves is discrete.
Equivalently, $G$ satisfies (UI) if and only if it satisfies
$(\mathrm{UI})_{\gamma}$ for any formal irreducible curve $\gamma$.
\end{pro}
As a byproduct of the Main Theorem, we obtain  Theorem \ref{teo:fd_discrete}.
We include here its proof for clarity, since it is not used in the proof of the Main Theorem.
\begin{proof}[proof of Theorem \ref{teo:fd_discrete}]
Suppose $G$ is finitely determined. Then $G$ satisfies (UI) by the Main Theorem and all
orbits of formal irreducible curves are discrete by Proposition \ref{pro:ui_discrete}.

Suppose that $G$ is non-finitely determined. There exists a sequence $(\phi_k)_{k \geq 1}$ of
elements in $G \setminus \{ \mathrm{Id} \}$ such that $j^{k} \phi_k \equiv  \mathrm{Id}$ for any
$k \geq 1$. Given any irreducible curve $\gamma$,  we have
$\lim_{k \to \infty} \phi_k (\gamma)=\gamma$.  The proof of Lemma \ref{lem:easy} implies that
the $G$-orbit of $\gamma$ is non-discrete for any generic irreducible curve $\gamma$.
\end{proof}
Next, we see  that the uniform intersection property is invariant under finite extensions.
This result has a technical interest in the sequel.
\begin{pro}
\label{pro:redfig}
Let $G$ be a subgroup of $\diffh{}{2}$ and $\gamma$ an irreducible formal curve.
Consider a finite index subgroup $H$ of $G$. Then $G$ satisfies
$(\mathrm{UI})_{\gamma}$ if and only if $H$ does.
\end{pro}
\begin{proof}
The sufficient condition is obvious. Suppose that $H$ satisfies $(\mathrm{UI})_{\gamma}$.
Then the $H$-orbit of $\gamma$ is discrete and closed by
Proposition \ref{pro:closure}. Since the $G$-orbit of $\gamma$ is a finite union of discrete closed
sets, it is discrete and closed.
\end{proof}
The next result is a corollary of Propositions \ref{pro:ui_discrete} and \ref{pro:redfig}.
\begin{cor}
\label{cor:redfi}
Let $G$ be a subgroup of $\diffh{}{2}$. Let $H$ be a finite index subgroup of $G$. Then $G$ satisfies
(UI) if and only if $H$ satisfies (UI).
\end{cor}
\begin{rem}
It is straightforward to show that $G$ satisfies (UI) if and only if a finite index subgroup $H$ does
for any dimension.
\end{rem}
\subsection{Reduction via blow-up}
\label{sec:blow}
In this section we will obtain further reductions, since we will get  simpler
expressions for the elements of a subgroup $G$ of  $\diffh{}{2}$ and a formal
irreducible curve $\gamma$ by considering blow-ups of
the infinitely near points of $\gamma$ (Propositions \ref{pro:blowfd}, \ref{pro:blowUI}
and \ref{pro:tame}).

First, we consider the problem of lifting vector fields and diffeomorphisms by the blow-up maps.
We will use the notations in section \ref{sec:action}.
\begin{rem}
\label{rem:lift}
Let $\gamma$ be a formal irreducible curve.
Let $X \in  \hat{\mathfrak X} \cn{2}$. We can lift this vector field to $\pi_{1}^{-1} (0,0)$ as a transversally formal
vector field. Moreover if the tangent line $\ell$ of $\gamma$ at $(0,0)$ is $j^{1} X$-invariant then $X$ can be lifted
to a formal vector field $\overline{X}_1 \in \hat{\mathfrak X} ({\mathcal X}_1, p_1)$.
Analogously if $p_j$ is a singular point of $\overline{X}_{j}$ for $1 \leq j <k$, we can lift $X$
to a formal vector field $\overline{X}_k$ at $p_k$. Moreover $\overline{X}_k$ belongs to
$\hat{\mathfrak X} ({\mathcal X}_k, p_k)$ if $p_k$ is a singular point of $\overline{X}_k$.
In a similar way if $\ell$ is $j^{1} \phi$-invariant for $\phi \in \diffh{}{2}$, we can lift $\phi$ to an element
$\tau_1 (\phi)$ of $\widehat{\mathrm{Diff}}({\mathcal X}_1, p_1)$ and so on.
\end{rem}

\begin{defi}
Let $G$ be a subgroup of $\diffh{}{2}$. We define $G_{\gamma, k}$ as the subgroup of $G$ whose
elements $\phi$ satisfy that $\phi(\gamma)$ and $\gamma$ share the first $k$ infinitely near points.
We define $G_{\gamma} =  \{ \phi \in G : \phi (\gamma) = \gamma \}$.
\end{defi}

Every element $\phi$ of $G$ induces an action in $\pi_{1}^{-1}(p_0)$ by a Mobius transformation.
Indeed $\phi$ belongs to $G_{\gamma,1}$ if and only if $p_1$ is a fixed point of such an action.
Thus we can lift $G_{\gamma, 1}$ to obtain a subgroup
$\tilde{G}_{\gamma, 1}$ of $\widehat{\mathrm{Diff}} (\widetilde{\mathbb C}^{2}, p_1)$.
Let $\tau_1 : G_{\gamma, 1} \to \tilde{G}_{\gamma, 1}$ be the group isomorphism
sending every element of $G_{\gamma, 1}$ to its lift in $\widehat{\mathrm{Diff}} (\widetilde{\mathbb C}^{2}, p_1)$.
All elements of $\tau_{1}(G_{\gamma, k})$ fix the first $k-1$ infinitely near points of $\gamma_1$.
By replacing $G_{\gamma. 1}$ and $\gamma$ with $\tau_1 (G_{\gamma, 2})$
and $\gamma_1$ respectively we obtain a group $\tilde{G}_{\gamma, 2}$ of formal local diffeomorphisms in a
neighborhood of the second infinitely near point $p_2$ of $\gamma$
and a group isomorphism
$\tau_{2} : G_{\gamma, 2} \to \tilde{G}_{\gamma, 2}$. By repeating this construction with
$\tau_{k}(G_{\gamma. k+1})$ we obtain a group $\tilde{G}_{\gamma, k+1}$ of formal local diffeomorphisms
and a group isomorphism $\tau_{k+1}:  G_{\gamma, k+1} \to \tilde{G}_{\gamma, k+1}$  for $k \geq 1$.

\begin{rem}
\label{rem:redk}
Since $(m_{k} (\phi (\gamma)))_{k \geq 0}$ does not depend on $\phi \in \diffh{}{2}$, the set
$\{ (\phi(\gamma), \gamma) : \phi \in G \setminus G_{\gamma, k} \}$ is bounded for any
$k \in {\mathbb N}$ by Equation (\ref{equ:inp}).
\end{rem}

Next let us show that the properties (FD) and (UI) for subgroups of $\diffh{}{2}$ are invariant by blow-up.
 \begin{pro}
 \label{pro:blowfd}
 Let $G$ be a subgroup of $\diffh{}{2}$. Suppose that all elements of $j^{1} G$ fix a direction $\ell$.
 Then $G$ has the finite determination property if and only if $\tilde{G}_{\ell, 1}$ has the finite determination
 property.
 \end{pro}
 \begin{proof}
 Suppose $\ell = \{y=0\}$ without lack of generality.
 Suppose that $G$ does not satisfy (FD).
 Then there exists a sequence $(\phi_{k})_{k \geq 1}$ in $G \setminus \{\mathrm{Id}\}$
 such that $j^{k} \phi_{k} = \mathrm{Id}$ for any $k \in {\mathbb N}$.
 Let   $\tilde{\phi}_{k} (x,t) = ( \tilde{a}_{k}(x,t), \tilde{b}_{k}(x,t) )$ be the lift of $\phi_{k}(x,y) = (a_{k}(x,y), b_{k}(x,y))$
 to $\tilde{G}_{\ell, 1}$. We have
 \[ \tilde{\phi}_{k} (x,t) = \left( \tilde{a}_{k}(x,t), \tilde{b}_{k}(x,t) \right) =
 \left( a_{k}(x,xt), \frac{b_{k}(x,xt)}{a_{k}(x,xt)} \right) \]
 for $k \in {\mathbb N}$.
 We obtain $j^{k} \tilde{a}_{k}(x,xt) = x$, $j^{k-1} \tilde{b}_{k}(x,t)= t$ and $\tilde{\phi}_{k} \not \equiv \mathrm{Id}$
 for any $k \in {\mathbb N}$. Therefore $\tilde{G}_{\ell, 1}$ does not satisfy (FD).

 Suppose that $\tilde{G}_{\ell, 1}$ does not satisfy (FD).
 There exists a sequence $(\phi_{k})_{k \geq 1}$ in $G \setminus \{\mathrm{Id}\}$ such that
 $\phi_{k} \neq \mathrm{Id}$ and $j^{k} \tilde{\phi}_{k} \equiv \mathrm{Id}$ for any $k \in {\mathbb N}$. Since we have
 \[ \phi_{k}(x,y) =
 \left( \tilde{a}_{k} \left( x , \frac{y}{x} \right), \tilde{a}_{k} \tilde{b}_{k}  \left( x , \frac{y}{x} \right) \right), \]
 we obtain $j^{[\frac{k+1}{2}]-1} a_{k} = x$ and $j^{[\frac{k+2}{2}]-1} b_{k} = y$. Thus
 $G$ does not hold (FD).
 \end{proof}

 \begin{pro}
 \label{pro:blowUI}
 Let $G$ be a subgroup of $\diffh{}{2}$.
 Consider a formal irreducible curve $\gamma$ and $1 \leq j \leq k$.
 Then $G$ satisfies $(\mathrm{UI})_{\gamma}$ if and only if $\tau_{j} ({G}_{\gamma, k})$ satisfies
 $(\mathrm{UI})_{\gamma_j}$.
 \end{pro}
 \begin{proof}
Fix $k \geq 1$. The group $G$ satisfies $(\mathrm{UI})_{\gamma}$
if and only if $G_{\gamma, k}$ does by Remark \ref{rem:redk}. Equation (\ref{equ:inp}) implies
\[ (\phi (\gamma), \gamma) = m_{0} (\gamma)^{2} + (\tau_{1}(\phi) (\gamma_1), \gamma_1)  \]
for any $\phi \in G_{\gamma, 1}$. Thus $G$ satisfies $(\mathrm{UI})_{\gamma}$ if and only if
$\tau_{1} ({G}_{\gamma, k})$ satisfies $(\mathrm{UI})_{\gamma_1}$. By iterating the previous argument
we obtain that $G$ satisfies $(\mathrm{UI})_{\gamma}$ if and only if
$\tau_{j} ({G}_{\gamma, k})$ satisfies $(\mathrm{UI})_{\gamma_j}$ for all $1 \leq j \leq k$.
\end{proof}

Next, we see that blow-ups can be used to reduce the proof of property $\mathrm{(UI)}_{\gamma}$
to simpler settings.

 \begin{defi}
 Let $G$ be a finitely determined subgroup of $\diffh{}{2}$ and $\gamma$ be a formal smooth curve.
 We say that the pair $(G,\gamma)$ is {\it tame} if there exists a formal smooth curve $\alpha$
 transversal to $\gamma$ that is $G$-invariant and the tangent direction of $\gamma$ at $(0,0)$
 is $j^{1} G$-invariant.
 \end{defi}

Tame pairs are easier to handle. Indeed $j^{1} G$ is diagonalizable and thus $G'$ is contained in
$\diffh{1}{2}$. Since a (FD) subgroup of $\diffh{1}{n}$ is nilpotent (Lemma \ref{lem:fdinil}),
it follows that $G$ is solvable. This allows to use the properties of solvable subgroups of $\diffh{}{2}$
and solvable Lie subalgebras of $\hat{\mathfrak X} \cn{2}$ in next section
to prove the Main Theorem.

 \begin{pro}
 \label{pro:tame}
 Suppose that $G$ satisfies property $(\mathrm{UI})_{\gamma}$ for any tame pair $(G,\gamma)$.
 Then any finitely determined subgroup of $\diffh{}{2}$ satisfies (UI).
 \end{pro}
 \begin{proof}
 Consider a finitely determined subgroup $G$ of $\diffh{}{2}$ and a formal irreducible curve $\gamma$.
 It suffices to prove $(\mathrm{UI})_{\gamma}$ by Proposition \ref{pro:ui_discrete}.

By choosing $k \in {\mathbb N}$ sufficiently big, we can suppose, up to make
$k$ blow-ups following the first $k$ infinitely near points of $\gamma$, that
$\gamma_{k}$ is smooth and
transversal to a $G$-invariant smooth divisor $\alpha$.
By replacing $G$ with $G_{\gamma, k+1}$ and applying $k$ times Proposition
\ref{pro:blowfd}, we obtain that $\tau_{k}(G_{\gamma, k+1})$ is finitely determined.
Moreover $(\tau_{k}(G_{\gamma, k+1}), \gamma_{k})$ is tame by construction.
By hypothesis $\tau_{k}(G_{\gamma, k+1})$
satisfies $(\mathrm{UI})_{\gamma_{k}}$.
We deduce that $G$ satisfies $(\mathrm{UI})_{\gamma}$ by Proposition \ref{pro:blowUI}.
\end{proof}
%
%
\section{Lie algebras of formal vector fields}
One of the ingredients in the proof of the Main Theorem is the classification of nilpotent and solvable
Lie subalgebras of $\hat{\mathfrak X} \cn{2}$ \cite{JR:arxivdl}.
The main idea of this section is that an infinite dimensional nilpotent (or solvable) Lie subalgebra of
$\hat{\mathfrak X} \cn{2}$ has a high degree of symmetry and very particular properties.
\begin{pro}
\label{pro:nla}
Let ${\mathfrak g}$ be a nilpotent subalgebra of $\hat{\mathfrak X} \cn{2}$. Then either ${\mathfrak g}$
is finite dimensional as a complex vector space or there exists $X \in {\mathfrak g}$ that has
a first integral in $\hat{\mathcal O}_{2} \setminus {\mathbb C}$ such that
\begin{equation}
\label{equ:nla}
{\mathfrak g} \subset \{ f X : f \in \hat{\mathcal O}_{2} \ \mathrm{and} \ X(f)=0 \} .
\end{equation}
In particular if ${\mathfrak g}$ is infinite dimensional then ${\mathfrak g}$ is abelian.
\end{pro}
\begin{rem}
\label{rem:first}
The properties of the subfields of $\hat{K}_2$ of first integrals of solvable Lie subalgebras of
$\hat{\mathfrak X} \cn{2}$ are fundamental to classify them \cite{JR:arxivdl}.
Let $X \in \hat{\mathfrak X} \cn{2} \setminus \{0\}$.
We denote ${\mathcal M} = \{ f \in \hat{K}_2 : X(f)=0 \}$. We have
\begin{itemize}
\item If ${\mathcal M} \cap \hat{\mathcal O}_{2} \neq {\mathbb C}$ then any element $g$ of ${\mathcal M}$
satisfies either $g \in \hat{\mathcal O}_{2}$ or $1/g \in \hat{\mathcal O}_{2}$.
Moreover ${\mathcal M} \cap \hat{\mathcal O}_{2}$ is equal to ${\mathbb C}[[f_0]]$ for some
$f_0 \in {\mathcal M} \cap {\mathfrak m}$.
\item If ${\mathcal M} \cap \hat{\mathcal O}_{2} = {\mathbb C}$
but  ${\mathcal M} \cap \hat{K}_{2} \neq {\mathbb C}$  then we get
${\mathcal M}= {\mathbb C} (f_0)$ for some
$f_0 \in {\mathcal M} \cap (\hat{K}_{2} \setminus {\mathbb C})$.
\end{itemize}
The previous results were proved by Mattei-Moussu \cite{MaMo:Aen} and Cerveau-Matei \cite{Ce-Ma:Ast}
respectively for the case $X \in \mathfrak{X} \cn{2}$. They can be easily adapted for formal vector fields
(cf. \cite[Proposition 5.1]{JR:solvable25}).
\end{rem}
Next lemma provides a hint of how the existence of first integrals has an influence on whether or not
a Lie subalgebra of $\hat{\mathfrak X} \cn{2}$ is finite dimensional.
\begin{lem}
\label{lem:nla}
Let $X \in \hat{\mathfrak X} \cn{2} \setminus \{0\}$ that has no first integral in
$\hat{\mathcal O}_{2} \setminus {\mathbb C}$. Then the complex vector space
\[ {\mathfrak g} := \{  f X : f \in \hat{K}_{2}, \ X(f)=0 \ \mathrm{and} \ f X \in \hat{\mathfrak X} \cn{2} \} \]
is finite dimensional.
\end{lem}
\begin{proof}
Denote ${\mathcal M} = \{ f \in \hat{K}_2 : X(f)=0 \}$.
We can suppose that $X$ has non-constant first integrals since otherwise
$\dim_{\mathbb C} {\mathfrak g} =1$. Hence we have
${\mathcal M} = {\mathbb C} (f_0)$ where $f_0 \not \in \hat{\mathcal O}_{2}$ and
$1/f_0 \not \in \hat{\mathcal O}_{2}$ by Remark \ref{rem:first}.

The formal vector field $X$ is of the form $g X'$ where the coefficients of $X'$ do not have a common
factor in $\hat{\mathcal O}_{2}$. It suffices to  show $\dim_{\mathbb C} V < \infty$ where
\[ V =  \left\{ f \in \hat{\mathcal O}_{2} : X \left( \frac{f}{g} \right) = 0 \right\}  . \]
Consider the map
\[
\begin{array}{ccc}
V & \stackrel{\Lambda}{\rightarrow} & j^{k} \hat{\mathcal O}_{2} \\
f & \mapsto & j^{k} f .
\end{array}
\]
Let us show $\mathrm{Ker} (\Lambda) = \{0\}$ if $k>>1$.
In this way we obtain
$\dim_{\mathbb C} V \leq \dim_{\mathbb C}  j^{k} \hat{\mathcal O}_{2} < \infty$.

Since every irreducible component of $f_0 = c$ for any $c \in {\mathbb C}$ is $X$-invariant,
there are infinitely many $X$-invariant formal irreducible curves.
Therefore there exists a dicritic divisor $D$ in the desingularization of
the dual form $\omega$ of $X$ ($\omega (X)=0$), i.e.
$D$ is not an invariant curve for the ``formal foliation" defined by the strict transform of
$\omega$.
The function $g$ vanishes along $D$ with order $m$ for some $m \in {\mathbb N}$.
Consider $k > m$. Any element $f$ of $\mathrm{Ker}(\Lambda)$ vanishes along $D$ with order at least
$m+1$ and then $f/g$ vanishes along $D$. Since $D$ is dicritic and $X(f/g)=0$, it follows $f/g =0$ and
then $f=0$.
\end{proof}
\begin{proof}[proof of Proposition \ref{pro:nla}]
Denote ${\mathcal M} = \{ f \in \hat{K}_2 : X(f)=0 \}$.
Suppose $\dim_{\hat{K}_{2}} ({\mathfrak g} \otimes_{\mathbb C} \hat{K}_2) = 2$.
Consider an element $X$ of ${\mathcal C}^{k} {\mathfrak g} \setminus \{0\}$ where $k+1 \geq 1$ is the
nilpotency class of ${\mathfrak g}$; it belongs to the center of ${\mathfrak g}$.
 Let $Y$ be an element of ${\mathfrak g}$ that is not of the form
$h X$ for some $h \in \hat{K}_{2}$. The elements of ${\mathfrak g}$ are of the form
\[ g X + h Y \]
where $g,h \in \hat{K}_2$. Since $X$ is in the center of ${\mathfrak g}$, we get
$[X,Y]=0$ and then $[X, gX + hY] = X(g) X + X(h) Y =0$. We deduce
$g, h \in {\mathcal M}$ for any $g X + hY \in {\mathfrak g}$. The property
\begin{equation}
\label{equ:itlie}
\underbrace{[Y, \cdots , [Y, gX+hY] ]  }_{m \text{ times}} = Y^{m} (g) X + Y ^{m} (h) Y
\end{equation}
and  ${\mathcal C}^{k+1} {\mathfrak g} = \{0\}$ imply
$Y^{k+1}(g)=Y^{k+1} (h) = 0$ for any  $g X + hY \in {\mathfrak g}$.
Consider the complex vector space
\[ V_j := \{ g \in \hat{K}_2:  X(g)=0 \ \mathrm{and} \ Y^{j}(g)=0 \} \]
for $j \in {\mathbb N}$. We showed that an element $g X + h Y$ of ${\mathfrak g}$
satisfies $g,h \in V_{k+1}$.
We have $V_{1} = {\mathbb C}$.  The linear map
\[
\begin{array}{ccc}
V_{j+1} & \stackrel{\Lambda_j}{\rightarrow} & V_j \\
g & \mapsto & Y(g)
\end{array}
\]
is well-defined since $X (Y(g)) = Y(X(g)) = Y(0) =0$ and $Y^{j} (Y(g)) =0$ for any $g \in V_{j+1}$.
The kernel of $\Lambda_{j}$ coincides with $V_1$ and hence
$\dim_{\mathbb C} V_{j+1} \leq \dim_{\mathbb C} V_1 + \dim_{\mathbb C} V_{j}$.
We obtain $\dim_{\mathbb C} V_j \leq j$ for any $j \in {\mathbb N}$ by induction on $j$.
Since $\dim_{\mathbb C} {\mathfrak g} \leq 2 \dim V_{k+1} \leq 2k+2$, the Lie algebra ${\mathfrak g}$
is finite dimensional.

Suppose now $\dim_{\hat{K}_{2}} ({\mathfrak g} \otimes_{\mathbb C} \hat{K}_2)  = 1$
and $\dim_{\mathbb C} {\mathfrak g} = \infty$.
Let $X$ be a non-vanishing element of the center of ${\mathfrak g}$.
Every element $Z$ of ${\mathfrak g}$ is of the form
$f X$ where $f \in \hat{K}_2$ and satisfies $[X, f X] = X(f) X =0$. We deduce
\[ {\mathfrak g} \subset \{ f X : f \in \hat{K}_{2} \ \mathrm{and} \ X(f) = 0 \} . \]
In particular ${\mathfrak g}$ is an abelian Lie algebra.
Then    $X$ has a first integral in
$\hat{\mathcal O}_{2} \setminus {\mathbb C}$ by Lemma \ref{lem:nla}.
In particular every first integral $h$ of $X$ in $\hat{K}_{2}$ satisfies either
$h \in \hat{\mathcal O}_{2}$ or $1/h \in  \hat{\mathcal O}_{2}$ by Remark \ref{rem:first}.
By replacing $X= a(x,y) \partial /\partial x + b(x,y) \partial /\partial x$ if necessary by an element
in ${\mathfrak g}$ whose
multiplicity $\min (m_{0} (a), m_{0}(b))$ at the origin is minimal,
every element of ${\mathfrak g}$ is of the form $(g/h) X$ where $g,h \in \hat{\mathcal O}_{2}$
and $m_{0}(g) \geq m_{0}(h)$.
Since $g/h$ is a first integral of $X$ and either $g/h$ or $h/g$ belongs to $\hat{\mathcal O}_{2}$,
it follows that $g/h \in \hat{\mathcal O}_{2}$. Hence we obtain Property (\ref{equ:nla}).
\end{proof}
Next, we consider general Lie subalgebras of $\hat{\mathfrak X} \cn{2}$.
\begin{pro}
\label{pro:sla}
Let ${\mathfrak g}$ be a Lie subalgebra of $\hat{\mathfrak X} \cn{2}$ such that
$\dim_{\mathbb C} {\mathfrak g} = \infty$. Then either ${\mathfrak g}$ is abelian or
$\dim_{\mathbb C} {\mathfrak g}'   = \infty$.
\end{pro}
\begin{proof}
Suppose $\dim_{\mathbb C} {\mathfrak g}'   < \infty$.  Our goal is showing that ${\mathfrak g}$ is abelian.
Let $\{ Y_1, \cdots, Y_p \}$ be a basis of ${\mathfrak g}'$.
Denote by ${\mathfrak h}$ the kernel of the linear map
\[
\begin{array}{ccc}
{\mathfrak g} & \rightarrow & ({\mathfrak g}')^{p} \\
Z & \mapsto &  ([Z,Y_1], \cdots, [Z,Y_p]) .
\end{array}
\]
It satisfies ${\mathfrak h} = \{ Z \in {\mathfrak g} : [Z, Y] =0 \ \forall Y \in {\mathfrak g}' \}$.
Moreover ${\mathfrak h}$ is an ideal of ${\mathfrak g}$ by Jacobi's formula.
Since $ ({\mathfrak g}')^{p}$ is finite dimensional, we obtain
$\dim_{\mathbb C} {\mathfrak g}/{\mathfrak h} < \infty$.
It is clear that ${\mathfrak h}$ is nilpotent and since $\dim_{\mathbb C} {\mathfrak h} = \infty$, it follows
that ${\mathfrak h}$ is abelian by Proposition \ref{pro:nla}.
Consider $\{ Z_1, \cdots, Z_q \} \subset {\mathfrak g}$
such that $\{ Z_1 + {\mathfrak h} , \cdots, Z_q + {\mathfrak h} \}$ is a basis of  ${\mathfrak g}/{\mathfrak h}$.
Denote by ${\mathfrak j}$ the kernel of the linear map
\[
\begin{array}{ccc}
{\mathfrak h} & \rightarrow & ({\mathfrak g}')^{q} \\
Z & \mapsto &  ([Z,Z_1], \cdots, [Z,Z_q]).
\end{array}
\]
The vector space ${\mathfrak j}$ satisfies $\dim_{\mathbb C} {\mathfrak h}/ {\mathfrak j} < \infty$.
It is clearly contained in the center $Z({\mathfrak g})$ of ${\mathfrak g}$ and hence
$\dim_{\mathbb C} {\mathfrak g}/ Z({\mathfrak g}) < \infty$ and $\dim_{\mathbb C} Z({\mathfrak g}) = \infty$.
Since $Z({\mathfrak g})$ is abelian and infinite dimensional, Proposition \ref{pro:nla} implies the existence of
non-trivial elements $X$ and $f X$ of $Z ({\mathfrak g})$ such that $X(f) =0$ and
$f \in \hat{\mathcal O}_{2} \setminus {\mathbb C}$. Given $Z \in {\mathfrak g}$ we have
$[Z, X]=0$ and $[Z, f X] = Z(f) X = 0$. Thus $f$ is a first integral of $Z$ and hence $Z$ is of the form
$g X$ where $g \in \hat{K}_2$. Since $0= [X, g X] = X(g) X$ for any $g X \in {\mathfrak g}$, we deduce that
${\mathfrak g}$ is abelian.
\end{proof}
\section{Algebraic properties of groups of local biholomorphisms}
We need to relate the properties of a subgroup $G$ of $\diffh{}{2}$ with the properties of its Lie algebra
of formal nilpotent vector fields. In this section we explain how to define such Lie algebras and the properties
of the Lie correspondence. The definitions and results
were introduced in \cite{JR:arxivdl} and \cite{JR:solvable25} and are included here for the sake of
simplicity.
\subsection{Zariski closure of a group of formal diffeomorphisms}
Given $\phi \in \diffh{}{n}$ we can interpret $\phi$ as a family $(\phi_{k})_{k \geq 1}$
of linear maps. More precisely $\phi_{k}$ is the element of $\mathrm{GL}({\mathfrak m}/ {\mathfrak m}^{k+1})$
defined by
\[
\begin{array}{ccc}
{\mathfrak m}/ {\mathfrak m}^{k+1} & \stackrel{\phi_{k}}{\rightarrow} & {\mathfrak m}/ {\mathfrak m}^{k+1} \\
f +  {\mathfrak m}^{k+1} & \mapsto & f \circ \phi + {\mathfrak m}/ {\mathfrak m}^{k+1}.
\end{array}
\]
We define $\pi_{k} : \diffh{}{n} \to D_k$ by the formula $\pi_{k} (\phi) = \phi_k$ where
\[ D_{k} := \{ A \in \mathrm{GL}({\mathfrak m}/ {\mathfrak m}^{k+1}) : A(f g) = A(f) A(g) \
\forall f,g \in {\mathfrak m}/ {\mathfrak m}^{k+1} \}. \]
Notice that $D_k$ is the group
of isomorphisms of the ${\mathbb C}$-algebra ${\mathfrak m}/ {\mathfrak m}^{k+1}$. It is an algebraic
subgroup of $\mathrm{GL}({\mathfrak m}/ {\mathfrak m}^{k+1})$
since the equation $A(f g) = A(f) A(g)$ is algebraic in the coefficients of the matrix $A$.
We have $D_k = \{ \phi_k : \phi \in \diffh{}{n} \}$ \cite[Lemma 2.1]{JR:solvable25}.

We have a natural map $\pi_{k,l}: D_{k} \to D_{l}$ for $k \geq l$ defined by $\pi_{k,l} (\phi_k) = \phi_l$.
The projective limit $\varprojlim_{k \in {\mathbb N}} D_{k}$ is called group of formal diffeomorphisms.
Given $f \in {\mathfrak m}$ then  $(\phi_{k}(f + {\mathfrak m}^{k+1}))_{k \geq 1}$ belongs to
$\varprojlim_{k \in {\mathbb N}} {\mathfrak m}/ {\mathfrak m}^{k+1}$.
We can interpret $(\phi_{k}(f + {\mathfrak m}^{k+1}))_{k \geq 1}$ as an element of ${\mathfrak m}$
since ${\mathfrak m}$ and $\varprojlim_{k \in {\mathbb N}} {\mathfrak m}/ {\mathfrak m}^{k+1}$
are naturally identified. Moreover the map
\[
\begin{array}{ccc}
\varprojlim_{k \in {\mathbb N}} D_{k} & \rightarrow & \diffh{}{n} \\
(\phi_{k})_{k \geq 1} & \mapsto & ( (\phi_{k}(z_{1} + {\mathfrak m}^{k+1}))_{k \geq 1}, \cdots,
(\phi_{k}(z_{n} + {\mathfrak m}^{k+1}))_{k \geq 1})
\end{array}
\]
is an anti-isomorphism of groups \cite[Lemma 2.2]{JR:solvable25}.

We define $G_k$ as the Zariski closure of $\pi_{k}(G)$ in the linear algebraic group $D_k$.
From the surjective nature of $\pi_{k,l} : \pi_{k}(G) \to \pi_{l}(G)$, it is possible to deduce that
$\pi_{k,l} : G_k \to G_l$ is a well-defined and surjective morphism of algebraic groups for all
$k \geq l \geq 1$ \cite[Lemma 2.5]{JR:solvable25}.
\begin{defi}
\cite{JR:arxivdl}
Let $G$ be a subgroup of $\diffh{}{n}$. We define the {\it Zariski closure} (or also {\it pro-algebraic closure})
$\overline{G}$ of $G$ as the group $\varprojlim_{k \in {\mathbb N}} G_{k}$. Indeed we have
\[ \overline{G} = \{ \phi \in \diffh{}{n} : \phi_k \in G_k \ \forall k \in {\mathbb N} \} . \]
We say that $G$ is pro-algebraic if $G = \overline{G}$.
\end{defi}
\begin{defi}
\label{def:krull}
Let $G$ be a subgroup of $\diffh{}{n}$. The closure of $G$ in the {\it Krull topology}
(${\mathfrak m}$-{\it adic topology}) is the subgroup of
$\diffh{}{n}$ consisting of the elements $\phi$ of  $\diffh{}{n}$ such that there exists $\eta(k) \in G$ satisfying
$j^{k} \phi = j^{k} \eta (k)$ for any $k \in {\mathbb N}$.
\end{defi}
\begin{rem}
$\overline{G}$ is closed in the Krull topology.
\end{rem}
\subsection{The Lie correspondence}
We denote by $L_k$ the Lie algebra of derivations of the ${\mathbb C}$-algebra
${\mathfrak m}/ {\mathfrak m}^{k+1}$. Let $X \in \hat{\mathfrak X} \cn{n}$.  The map $X_k$ defined by
\[
\begin{array}{ccc}
{\mathfrak m}/ {\mathfrak m}^{k+1} & \stackrel{X_{k}}{\rightarrow} & {\mathfrak m}/ {\mathfrak m}^{k+1} \\
f +  {\mathfrak m}^{k+1} & \mapsto & X(f) + {\mathfrak m}/ {\mathfrak m}^{k+1}.
\end{array}
\]
belongs to $L_k$. Analogously as for formal diffeomorphisms $\hat{\mathfrak X} \cn{n}$ can be identified
with $\varprojlim_{k \in {\mathbb N}} L_{k}$. Moreover $L_k$ is the Lie algebra of $D_k$ for any
$k \in {\mathbb N}$. Given $X=(X_k)_{k \geq 1} \in \varprojlim  L_{k}$ we can define
$\mathrm{exp}(X) = (\mathrm{exp}(X_k))_{k \geq 1} \in \varprojlim  D_{k}$.
The previous definition implies
\[ f \circ \mathrm{exp} (X) = \sum_{k=0}^{\infty} \frac{X^{k}(f)}{k!} \]
where $X^{0}(f)=f$ and $X^{k+1}(f) = X(X^{k}(f))$ for $k \geq 0$. In particular we obtain
\[ \mathrm{exp}(X)(z_1, \cdots, z_n) = \left(
\sum_{k=0}^{\infty} \frac{X^{k}(z_1)}{k!}, \cdots, \sum_{k=0}^{\infty} \frac{X^{k}(z_n)}{k!}
\right) . \]

Given a subgroup $G$ of $\diffh{}{n}$ we define the Lie algebra of $G$ (or $\overline{G}$)
as the Lie algebra $\varprojlim_{k \in {\mathbb N}} {\mathfrak g}_{k}$ where ${\mathfrak g}_{k}$
is the Lie algebra of $G_k$. It is clearly closed in the Krull topology (the definition is analogous to
Definition \ref{def:krull}).
The Lie algebra ${\mathfrak g}$ of $G$ satisfies
\begin{equation}
\label{equ:defliealg}
 {\mathfrak g} = \{ X \in \hat{\mathfrak X} \cn{n} : \mathrm{exp}(t X) \in \overline{G} \ \forall t \in
 {\mathbb C} \},
 \end{equation}
cf. Proposition 2 of \cite{JR:arxivdl}.
We define the connected component of identity
$\overline{G}_0= \varprojlim_{k \in {\mathbb N}} G_{k,0}$
where $G_{k,0}$ is the connected component of $\mathrm{Id}$ of $G_k$.
\begin{pro}[{\cite[Proposition 2]{JR:arxivdl}} {\cite[Proposition 2.3, Remark 2.9]{JR:solvable25}}]
\label{pro:lie}
Let $G$ be a subgroup of $\diffh{}{n}$.
Then $\overline{G}_0$ is a finite index normal pro-algebraic subgroup of $\overline{G}$
generated by $\mathrm{exp} ({\mathfrak g})$.
Moreover if $G$ consists of unipotent formal diffeomorphisms then
$\mathrm{exp} : {\mathfrak g} \to \overline{G}$ is a bijection and ${\mathfrak g}$ consists of formal
nilpotent vector fields.
\end{pro}
\begin{rem}[{cf. \cite{Ecalle, MaRa:aen}}]
The exponential provides a bijection between $\hat{\mathfrak X}_{N} \cn{n}$ and
$\diffh{u}{n}$.
\end{rem}
\begin{defi}
 Given $\phi \in \diffh{u}{n}$ we denote by $\log \phi$ the unique formal nilpotent vector field
such that $\phi = \mathrm{exp}(\log \phi)$. It is called the {\it infinitesimal generator} of $\phi$.
\end{defi}
\begin{rem}
\label{rem:explog}
We have $\overline{\langle \phi \rangle} = \{ \mathrm{exp} (t \log \phi) : t \in {\mathbb C} \}$ for any
$\phi \in \diffh{u}{2}$  \cite[Remark 2.11]{JR:solvable25}.
In particular if $\phi$ is a unipotent element of a subgroup $G$ of $\diffh{}{n}$, then
the formal vector field $\log \phi$ belongs to the Lie algebra
${\mathfrak g}$ of $G$ as a consequence of $\overline{\langle \phi \rangle} \subset \overline{G}$.
\end{rem}
\begin{rem}
\label{rem:invexplog}
Let $\phi \in \diffh{u}{2}$.
Consider a formal irreducible curve $\gamma$. Then $\gamma$ is $\phi$-invariant
if and only if is $\log \phi$-invariant. The necessary condition is obvious. Let us show the sufficient condition.
The group $H_{\gamma} = \{\eta \in \diffh{}{n} : \eta (\gamma)=\gamma\}$ is pro-algebraic
\cite[Remark 2.8]{JR:solvable25}.
Thus $\overline{\langle \phi \rangle} = \{ \mathrm{exp} (t \log \phi) : t \in {\mathbb C} \}$
is contained in $H_{\gamma}$.
Since the one parameter flow of $\log \phi$ preserves $\gamma$, it follows that $\gamma$ is
$\log \phi$-invariant.
\end{rem}

Instead of the derived series of groups of formal diffeomorphisms and Lie algebras of formal
vector fields, it is more natural to consider its projective versions. More precisely, we define
\[ \overline{G}^{(k)} = \{ \phi \in \diffh{}{n} : \phi_j \in G_{j}^{(k)} \ \forall j \in {\mathbb N} \}, \]
\[  \overline{\mathfrak g}^{(k)} = \{ X \in \hat{\mathfrak X} \cn{n} : X_j \in {\mathfrak g}_{j}^{(k)}
\ \forall j \in {\mathbb N} \} \]
and
\[ \overline{\mathcal C}^{k} G = \{ \phi \in \diffh{}{n} : \phi_j \in {\mathcal C}^{k}  G_{j} \ \forall j \in {\mathbb N} \}, \]
\[ \overline{\mathcal C}^{k}  {\mathfrak g} =
 \{ X \in \hat{\mathfrak X} \cn{n} : X_j \in {\mathcal C}^{k} {\mathfrak g}_{j}  \ \forall j \in {\mathbb N} \} \]
for $k \geq 0$. We have $\overline{G}^{(0)} = \overline{\mathcal C}^{0} G = \overline{G}$ and
$\overline{\mathfrak g}^{(0)} = \overline{\mathcal C}^{0}  {\mathfrak g}  = {\mathfrak g}$.
\begin{rem}[{\cite[Lemmas 5 and 6]{JR:arxivdl}} {\cite[Proposition 2.8]{JR:solvable25}}]
The group $\overline{G}^{(k)}$ is pro-algebraic and it is equal to the closure of
$(\overline{G})^{(k)}$ in the Krull topology for any $k \in {\mathbb N}$.
Analogously $\overline{\mathfrak g}^{(k)}$ is the closure in the Krull topology of ${\mathfrak g}^{(k)}$
for any $k \in {\mathbb N}$.
\end{rem}
 The Lie correspondence is well-behaved for the derivation.
 \begin{pro}[{\cite[Proposition 3]{JR:arxivdl}}]
 \label{pro:dera}
 Let $G$ be a subgroup of $\diffh{}{n}$ that is connected, i.e. $\overline{G} = \overline{G}_0$.
 Then $\overline{\mathfrak g}^{(k)}$ is the Lie algebra of $\overline{G}^{(k)}$ and
 $\overline{\mathcal C}^{k} {\mathfrak g}$ is the Lie algebra of $\overline{\mathcal C}^{k} G$
 for any $k \geq 0$.
 \end{pro}
 \begin{rem}
 \label{rem:niso}
 Since
\[ \overline{\pi_{k}(G^{(j)})} =  (\overline{\pi_{k}(G)})^{(j)}  = G_{k}^{(j)} \ \mathrm{and} \
\overline{\pi_{k}({\mathcal C}^{j} G )} = {\mathcal C}^{j} \overline{\pi_{k}( G )} = {\mathcal C}^{j}  G_{k}  \]
for $j \geq 0$ and $k \geq 1$ (cf. \cite[Corollary 1, p. 60]{Borel}),  the Zariski closure
$\overline{G^{(j)}}$ of $G^{(j)}$ is equal to $\overline{G}^{(j)}$
and $\overline{{\mathcal C}^{j} G} =\overline{\mathcal C}^{j} G$ for $j \geq 0$.
In particular $G$ is abelian (resp.  nilpotent, solvable) if and only if
$\overline{G}$ is abelian (resp. nilpotent, solvable).
Proposition \ref{pro:dera} implies that if $G$ is abelian (resp. nilpotent, solvable) then ${\mathfrak g}$ is
abelian (resp. nilpotent, solvable). The reciprocal also holds if $\overline{G} = \overline{G}_0$.
\end{rem}
\subsection{Dimension of groups of formal diffeomorphisms}
 The sequence $(\dim G_k)_{k \geq 1}$ is increasing \cite[Lemma 3.1]{JR:finite}. Thus we can define
 \begin{defi}
 Let $G$ be a subgroup of $\diffh{}{n}$. We define
 $\dim G = \lim_{k \to \infty} \dim G_k \in {\mathbb Z}_{\geq 0} \cup \{\infty\}$.
 We say that $G$ is {\it finite dimensional} if $\dim G < \infty$.
 \end{defi}
\begin{rem}[{\cite{JR:finite}}]
\label{rem:finite}
A subgroup $G$ of $\diffh{}{n}$ is finite dimensional if and only if its Lie algebra ${\mathfrak g}$
is finite dimensional.
\end{rem}
One of the major points in the proof of the Main Theorem is that the next result
allows us to reduce our study to infinite dimensional groups.
 \begin{teo}
\label{teo:UIfdim}
\cite[Theorem 1.5]{JR:finite}
Let $G$ be a finite dimensional subgroup of $\diffh{}{2}$. Then $G$ satisfies (UI).
\end{teo}
\section{Proof of the Main Theorem}
In this section we show that (FD) implies (UI) in dimension $2$.
We start showing the result for nilpotent groups (Proposition \ref{pro:UIn}) and then
we proceed with the general case.
\begin{pro}
\label{pro:UIn}
Let $G$ be a (FD) nilpotent subgroup of $\diffh{}{2}$.
Then $G$ satisfies (UI).
\end{pro}
\begin{proof}
Consider a formal irreducible curve $\gamma$.
Let us show that $G$ satisfies $(\mathrm{UI})_{\gamma}$.
By Propositions \ref{pro:blowfd} and  \ref{pro:blowUI} we can suppose that $\gamma$ is smooth up to replace
$\gamma$ and $G$ with $\gamma_{j}$ and $\tau_{j} ({G}_{\gamma, j})$ respectively for some $j \geq 1$.

We can suppose that $G$ is infinite dimensional by Theorem \ref{teo:UIfdim}.
Hence the Lie algebra ${\mathfrak g}$ of $G$ is infinite dimensional and nilpotent
by Remarks \ref{rem:finite} and \ref{rem:niso}.
We deduce that ${\mathfrak g}$ is abelian and such that
\[ {\mathfrak g} \subset \{ g X : g \in \hat{\mathcal O}_{2} \ \mathrm{and} \ X(g)=0 \} \]
for some $X \in {\mathfrak g} \setminus \{0\}$ that has a first integral in
$\hat{\mathcal O}_{2} \setminus {\mathbb C}$ by Proposition \ref{pro:nla}.
The group $\overline{G}_{0}$ is generated by $\mathrm{exp} ({\mathfrak g})$ (Proposition \ref{pro:lie})
and since ${\mathfrak g}$
is abelian, $\overline{G}_{0}$ is abelian and $\overline{G}_{0} = \mathrm{exp}({\mathfrak g})$.

Suppose that $\gamma$ is $X$-invariant. Then it is ${\mathfrak g}$-invariant and we
obtain $\phi  (\gamma) = \gamma$ for any $\phi \in \overline{G}_{0}$.
Since $\overline{G}_{0}$ is a finite index subgroup of $\overline{G}$ by Proposition \ref{pro:lie},
the $\overline{G}$-orbit of $\gamma$ is finite and both $\overline{G}$ and $G$ satisfy $(\mathrm{UI})_{\gamma}$.
Thus we suppose that $\gamma$ is not $X$-invariant from now on.

Let $(p_{k})_{k \geq 1}$ be the sequence of infinitely near points of $\gamma$ (see section \ref{sec:blow}).
Consider the lift $\overline{X}_{k}$ of $X$ at $p_{k}$ for $k \geq 0$ (see Remark \ref{rem:lift}).
Since $\gamma$ is not $X$-invariant,
there exists $a \geq 1$ such that $p_{a}$ is the first infinitely near point of $\gamma$ satisfying that
$\overline{X}_{a}$ is non-singular at $p_{a}$.
Moreover $\overline{X}_{a}$ preserves the irreducible components of the divisor
$(\pi_1 \circ \hdots \circ \pi_a)^{-1} (0,0)$ of the
blow-up process passing through
$p_{a}$ (cf. section \ref{sec:blow}).
Since $\overline{X}_{a}$ is non-singular at $p_{a}$, there is a unique such component $D$.
Moreover since $\gamma$ is smooth, $\gamma_{a}$ is transversal to $D$ at $p_{a}$.
Hence there exists a formal coordinate system $(x,y)$ centered at $p_{a}$ such that
$\overline{X}_{a} = \frac{\partial}{\partial y}$, $D= \{x=0\}$ and $\gamma_{a} = \{ y=0 \}$.

We denote $J=G_{\gamma, a}$ and $H=J \cap \overline{J}_0$.
Let ${\mathfrak h}$ be the Lie algebra of $J$.
The inclusion $J \subset G$ implies ${\mathfrak h} \subset {\mathfrak g}$ and in particular
${\mathfrak h}$ is abelian.
As a consequence of the characterization of
${\mathfrak h}$ given by Equation
(\ref{equ:defliealg}), $p_j$ is a singular point of $\overline{(g X)}_j$
for all $g X \in {\mathfrak h}$ and $0 \leq j \leq a$.
Since $p_j$ is a singular point of $\overline{X}_j$, it is also singular for the lift of
any element of ${\mathfrak g}$ at $p_j$ for every $0 \leq j < a$.
The lift $\overline{(g X)}_a$ of $g X \in {\mathfrak g}$ is singular at $a$ if and only if
$g$ belongs to the maximal ideal ${\mathfrak m}$ of  $\hat{\mathcal O}_{2}$.
The previous discussion implies
\[ {\mathfrak h} \subset  \{ g X  \in {\mathfrak g} : g \in \hat{\mathcal O}_2 \cap {\mathfrak m}  \} .
 \]
Since $\overline{J}_0$ is a finite index normal subgroup of $\overline{J}$,
$H$ is a finite index normal subgroup of $J$.
The Zariski closure $\overline{H}$ of $H$ is a finite
index normal subgroup of $\overline{J}$ \cite[Lemma 2.4]{JR:finite}.
A finite index subgroup of $\overline{J}$ always contains $\overline{J}_0$ \cite[Lemma 2.1]{JR:finite}.
Since $\overline{J}_0$ is pro-algebraic by Proposition \ref{pro:lie}, we obtain
$\overline{J}_{0} \subset  \overline{H} \subset \overline{J}_0$ and then
$ \overline{H} = \overline{J}_0$. In particular we deduce $\overline{H} = \overline{H}_{0}$.
The property $\overline{H}=  \overline{J}_0$ implies that  the group
$\overline{H}$ is generated by $\mathrm{exp} ({\mathfrak h})$ (Proposition \ref{pro:lie}).
Since ${\mathfrak h}$ is abelian, the groups $H$ and $\overline{H}$ are abelian,
$\overline{H} = \mathrm{exp}({\mathfrak h})$ and in particular
$H \subset \mathrm{exp}({\mathfrak h})$.
Thus any $\phi \in \tau_{a} (H)$ is of the form
\[ \phi(x,y) =\mathrm{exp} \left( f(x) \frac{\partial}{\partial y} \right) =(x,y+f(x)) \]
where $f(x) \in {\mathbb C}[[x]] \cap (x)$ and
we obtain $(\phi (\gamma_{a}), \gamma_{a})= m_{0}(f)$.
Since $\tau_{a} (H)$ satisfies (FD) by Proposition \ref{pro:blowfd}, we deduce that
$\tau_{a} (H)$ satisfies $(UI)_{\gamma_{a}}$.
Therefore $H$ satisfies $(UI)_{\gamma}$ by Proposition \ref{pro:blowUI}.
Since $H$ is a finite index subgroup of $J$, the group $J$ satisfies  $(UI)_{\gamma}$ by Proposition \ref{pro:redfig}.
Finally $G$ satisfies  $(UI)_{\gamma}$ by Remark \ref{rem:redk}.
\end{proof}
\begin{lem}
\label{lem:fdinil}
Let $G$ be a (FD) subgroup of $\diffh{1}{n}$. Then $G$ is nilpotent.
\end{lem}
\begin{proof}
Let $\phi, \eta \in G$ such that $j^{k} \phi = \mathrm{Id}$. Then we have $j^{k+1} [\phi, \eta] = \mathrm{Id}$.
As a consequence ${\mathcal C}^{k} G$ is contained in $\{ \phi \in G : j^{k+1} \phi = \mathrm{Id} \}$.
Since $G$ is (FD), the latter group is trivial for $k>>1$ and hence $G$ is nilpotent.
\end{proof}
The following result is an immediate corollary of Proposition  \ref{pro:UIn} and Lemma \ref{lem:fdinil}.
\begin{cor}
\label{cor:uiUI}
Let $G$ be a (FD) subgroup of $\diffh{1}{2}$. Then $G$ satisfies (UI).
\end{cor}
\begin{proof}[proof of the Main Theorem]
The property (UI) implies the finite determination property by Lemma \ref{lem:easy}.

Let us show the necessary condition. It suffices to show the property $(\mathrm{UI})_{\gamma}$
for a tame pair $(G, \gamma)$ by Proposition \ref{pro:tame}.
Since $G \cap \diffh{1}{2}$ is finitely determined, it satisfies
$(\mathrm{UI})_{\gamma}$  by Corollary \ref{cor:uiUI}.
In particular there exists $M \in {\mathbb N}$ such that $(\phi (\gamma), \gamma) < M$ for
any $\phi \in (G \cap \diffh{1}{2}) \setminus G_{\gamma}$.
We deduce that $G_{\gamma, M} \cap \diffh{1}{2}$ is contained in $G_{\gamma}$.
As a consequence and up to replace
$G$ with $G_{\gamma, M}$ we can suppose
that $(G,\gamma)$ is tame and $G \cap \diffh{1}{2} \subset G_{\gamma}$ by Remark \ref{rem:redk}.
Moreover up to replace $G$ with its finite index subgroup $G \cap \overline{G}_0$ we can suppose
$\overline{G} = \overline{G}_{0}$ by Proposition \ref{pro:redfig}.

Since $(G,\gamma)$ is tame, the group $j^{1} G$ is diagonalizable and in particular
the derived group $G'$ is contained in $\diffh{1}{2}$.
The group $G'$ is nilpotent by Lemma \ref{lem:fdinil}.
We can suppose that $G$ is non-abelian and infinite dimensional by
Proposition \ref{pro:UIn} and Theorem \ref{teo:UIfdim} respectively.
Since $\overline{G} = \overline{G}_0$, Remark \ref{rem:niso} implies that ${\mathfrak g}$ is non-abelian.
Such a property together with $\dim_{\mathbb C} {\mathfrak g} = \infty$ (Remark \ref{rem:finite}) imply
$\dim_{\mathbb C} {\mathfrak g}^{(1)}= \infty$ by Proposition \ref{pro:sla}.
The Lie algebra ${\mathfrak h}$ of $G'$ is equal to $\overline{\mathfrak g}^{(1)}$
by Remark \ref{rem:niso} and Proposition \ref{pro:dera}.
The Lie algebra ${\mathfrak h}$ is nilpotent (Remark \ref{rem:niso}) and satisfies
$\dim_{\mathbb C} {\mathfrak h} \geq \dim_{\mathbb C} {\mathfrak g}^{(1)}= \infty$.
Hence we obtain
\[ {\mathfrak h} \subset  \{ f X : f \in \hat{\mathcal O}_{2} \ \mathrm{and} \ X(f) = 0 \}  \]
for some $X \in {\mathfrak h} \setminus \{0\}$ such that $X$ has a first integral $g$ in
$\mathfrak{m} \setminus \{0\}$ by Proposition \ref{pro:nla}.
The vector field $X$ is of the form $h X'$ where $h \in \hat{\mathcal O}_2$ and the coefficients of
$X'$ have no common factors in $\hat{\mathcal O}_2$.
Given a non-trivial element $f X$ of ${\mathfrak h}$, there are finitely many $f X$-invariant curves.
Indeed the formal irreducible $fX$-invariant curves
are the curves $h_j=0$ ($1 \leq j \leq k$) where $f g h= h_{1}^{n_1} \cdots h_{k}^{n_k}$ is the
irreducible decomposition of $f g h$ in $\hat{\mathcal O}_2$.
Choose $\phi \in G' \setminus \{\mathrm{Id}\}$.
The set of formal irreducible $\phi$-invariant curves coincides with the set of
formal irreducible $\log \phi$-invariant curves by Remark \ref{rem:invexplog}.
Since $\log \phi$ belongs to ${\mathfrak h} \setminus \{0\}$ by Remark \ref{rem:explog},
there are finitely many formal irreducible $\phi$-invariant curves and hence
finitely many formal irreducible $G'$-invariant curves.

Since $G' \subset \diffh{1}{2}$ and $\gamma$ is $G \cap \diffh{1}{2}$-invariant,
the curve $\gamma$ is $G'$-invariant. Since $G'$ is a normal subgroup of $G$, every curve in
the $G$-orbit of $\gamma$ is also $G'$-invariant.
Therefore the $G$-orbit of $\gamma$ is finite and it is obvious that $G$ satisfies
$(\mathrm{UI})_{\gamma}$.
\end{proof}

\bibliography{rendu}
\end{document}